\documentclass{amsart}

\usepackage{amsmath, amssymb, amsthm, xcolor}
\usepackage[all]{xy}
\usepackage{ulem}	
\usepackage{enumerate}

\newtheorem{thm}{Theorem}
\newtheorem{lemma}[thm]{Lemma}
\newtheorem{notation}[thm]{Notation}
\newtheorem{prop}[thm]{Proposition}

\newtheorem{conj}[thm]{Conjecture}
\newtheorem{ques}[thm]{Question}
\theoremstyle{definition}
\newtheorem{defn}[thm]{Definition}
\newtheorem{eg}[thm]{Example}
\newtheorem{remark}[thm]{Remark}
\numberwithin{thm}{section}

\DeclareMathOperator{\depth}{depth}

\newcommand{\m}{\mathfrak{m}}

\renewcommand{\phi}{\varphi}

\DeclareMathOperator{\charac}{char}

\renewcommand{\to}{\longrightarrow}
\newcommand{\x}{\mathbf{x}}

\title{Numbers of Generators of Perfect Ideals}
\author{Raymond C Heitmann}
\date{\today}

\begin{document}

\maketitle

\begin{abstract}
This article is concerned with bounds on the number of generators of perfect ideals $J$ in regular local rings $(R,\m)$.
If $J$ is sufficiently large modulo $\m^n$, a bound is established depending only on $n$ and the projective dimension of $R/J$.
More ambitious conjectures are also introduced with some partial results.
\end{abstract}

This work is inspired by an article written by Ma \cite {M} on Lech's Conjecture, though ultimately it has no direct connections to any work on that conjecture.
Ma demonstrated that Lech's Conjecture would be established if one could establish a certain weaker form of the following conjecture on multiplicities. (Note that Ma did not state this as a conjecture.)
\begin{conj}
Let $(R,\m)$ be a complete local ring and let $J$ be a perfect ideal of $R$.
Then $e(R/J)\geq e(R)$.
\end{conj}

There is no evidence counter to this conjecture, but it seems rather difficult to attack.  
The basic premise behind it is that making an ideal perfect somehow makes it ``small".
Of course, literally the notion of small that we need is that $e(R/J)$ is sufficiently large, but there are other notions of smallness that may be useful - either $J$ has relatively few generators or $J$ is contained in a large power of the maximal ideal.
A natural approach is to blend these two things.

The problem is trivial when the height of $J$ is less than 2.
When the height of $J$ equals $2$, we have the Hilbert -Burch Theorem \cite [Theorem 20.15]{E}.
This tells us that if $J\nsubseteq \m^n$ for some integer $n$, $\mu(J)\leq n$.
However, if the height of $J$ is greater than two, essentially nothing is known.
Our goal then is to try to figure out what the results might be and attempt to prove them.
The author begins with a very optimistic premise, namely that perfect ideals in general resemble perfect ideals in regular local rings.
This is at least true when the height of $J$ is at most two.
Thus one would simply develop the theory for regular local rings and then extend it to the more general case.
At this time, this author has no plan for how to make these extensions when the regular local ring theory is available and frankly believes there are others more capable of doing so than he.
In any case, the logical first step is the development of a satisfactory theory for regular local rings.
Accordingly, the true subject of this article is handling the case where $R$ is a regular local ring.
The hoped for conclusion for RLR theorems is relatively obvious, the hypothesis less so.
We note that if $\dim R=d$, then $\mu(\m^{n-1})=\begin{pmatrix} n+d-2 \\ d-1 \end{pmatrix}$ and so there is a natural prospective upper bound.
It is relatively easy to prove that this bounds $\mu(J+\m^n/\m^n)$ (see Proposition~\ref{g=d}) but it need not bound $\mu(J)$.
Moreover, a true analogue of Hilbert-Burch should not depend on the dimension $d$ of $R$ but only on the height $g$ of $J$.
There are two somewhat obvious conjectures which are generalizations of Hilbert-Burch and are stated here for regular rings.
To this author, the word conjecture is more akin to a suggestion for a starting point and not an expectation of what is likely to be true.  
In fact, neither of these conjectures is correct in the form stated here.
An example demonstrating this will be offered in Section 2.
However, at that time, we shall offer modifications of the conjectures which we hope are true.
We shall also present our main theorem, Theorem~\ref{main}, at that time, a theorem which gives a bound which coincides with Hilbert-Burch in the $g=2$ case, but which is higher than the bounds we are striving for.
Finally in Section 3, we shall present a few positive results in the $g=3$ case for small $n$.

\begin{conj}\label{E1}
 Let $(R,\m)$ be a local ring and suppose $J$ is a perfect ideal of grade $g>0$.
Then $\mu(J+\m^n/\m^n)\leq \begin{pmatrix} g+n-2\\ g-1 \end{pmatrix}$.
\end{conj}

\begin{conj}\label{E2}
Let $(R,\m)$ be a regular local ring and suppose $J$ is a perfect ideal of grade $g>1$.
If $\mu(J+\m^n/\m^n)\geq \begin{pmatrix} g+n-3\\ g-2 \end{pmatrix}$, then $\mu(J)\leq \begin{pmatrix} g+n-2\\ g-1 \end{pmatrix}$.
\end{conj}

\section{notation}

Since the number of generators of ideals and modules is not affected by faithfully flat extensions, we can reduce our problem to the consideration of complete local rings with algebraically closed residue fields.
None of our proofs actually require an algebraically closed residue field and so we will not assume it.
However, an infinite residue field will often be needed and so we will make that a standard assumption.
Accordingly, throughout this article unless otherwise specified, $(R,\m,k)$ will be a complete regular local ring with an infinite residue field.
In general, we will not assume that $R$ is equal characteristic or unramified.

An ideal $I$ in a local ring $S$ is perfect if it has finite projective dimension and if $g$, the grade of $I$, equals the projective dimension of $S/I$.
Using the Auslander-Buchsbaum Formula, this implies that $\depth S/I=\depth S - g$ \cite [Theorem 19.9]{E}.
Since all ideals in regular local rings have finite projective dimension, it is easy to see that an ideal $I$ in a regular local ring is perfect if and only if $R/I$ is Cohen-Macaulay.
Throughout this article, $J$ will be a perfect ideal of grade $g$ in $R$.
Of course, $J$ also has height $g$.

We need a generalization of the standard bases of Hironaka, one that is applicable to the mixed characteristic case.
Standard bases were introduced in \cite{H} and developed more thoroughly by Becker in \cite{B}.
Throughout we let $A$ be a coefficient ring for $R$ (either a field or a DVR with maximal ideal $pA$) and let $x_1,\dots,x_d$ be a regular system of parameters.
If $A$ is a field, we have the usual setting $R=A[[x_1,\dots,x_d]]$, but the mixed characteristic case is more complicated and we do not assume $x_1=p$ even in the unramified case, the only case in which this is even possible.
Each element in $R$ can be viewed as a power series in the $x_i$'s with coefficients in $A$ such that each coefficient is either zero or a unit.
To see how to do this, given any $r\in R$, we can find an element $a\in A$ such that $r-a\in (x_1,\dots,x_d)R$.
In the equal characteristic case, $a$ is unique, but this is not true in the mixed characteristic case.
If $r\in (x_1,\dots,x_d)R$, we can choose $a=0$ and insist on doing so.
Thus, $a$ is either zero or a unit in $A$.
Now we have $r=a+x_1r_1+\dots+x_dr_d$.
We repeat the process with each $r_i$ and recursively obtain our power series representation.
Of course, if $A$ is not a field, this representation will not be unique.
Moreover it is not even true that the sum of two representations will be a representation.

We almost always\footnote{In the proof of Proposition~\ref{Lemma 487}, we use the more traditional lexicographic order without the sum term.} use a somewhat novel order on monomials.
We order the monomials $x_1^{e_1}\dots x_d^{e_d}$ using lexicographic order on the $(d+1)$-tuples \newline $(e_d,\dots,e_{g+1},\sum_{i=1}^ge_i,e_g,\dots,e_1)$.
The smallest monomials are then those which are not contained in the ideal $(x_{g+1},\dots,x_d)$; among those, the first criterion is degree.
Letting $T=\{t_\alpha\}$ denote the set of monomials suitably ordered, we note that $1<t$ for all $t\neq 1\in T$ and $t_1<t_2$ implies $st_1<st_2$ for all $s\in T$.
Thus our order is, by definition, an admissible term order on $T$.
Now we can express any element $f\in R$ as
$f=\sum c(t_\alpha,f)t_\alpha$ with each $c(t_\alpha,f)$ either zero or a unit in $A$.
Again we note that we must choose the set $\{c(t_\alpha,f)\}$ - it is not always unique.
Following \cite{B}, we set $T(f)=\{\alpha\mid c(t_\alpha,f)\neq 0\}$ and let $LT(f)$ be the least element in $T(f)$.
$T(f)$ may depend upon our representation.
However, as $t_\alpha\neq\sum_{\beta, t_\beta\nmid t_\alpha}c_\beta t_\beta$, $LT(f)$ is independent of our representation.

We will need a version of the Hironaka Theorem, which is found in \cite{B}.

\begin{thm}
(Hironaka Theorem) Let $<$ be an admissible order on $T$ and $I$ an ideal in a complete regular local ring.
Then there exists a finite set $S\subset I$ such that for every $f\in I$, there exists $g\in S$ with $LT(g)|LT(f)$.
Any such $S$ is a basis of $I$, and it is then called a standard basis of $I$ (wrt $<$).
If $f\in R$ and $\{g_1,\dots,g_m\}$ is a standard basis of $I$, then there exists $r\in R$ and a representation for $r$ with the properties that\begin{enumerate}
  \item there exist $q_1,\dots,q_m\in R$ with $f=\sum_{i=1}^mg_iq_i+r$, and
  \item for all $s\in LT(S)$, if $t\in T(r)$, then $s\nmid t$.
\end{enumerate}
\end{thm}

For the proof of the theorem, we refer to Becker \cite{B}.
The version stated here differs from that which appears in \cite{B} in two respects.
We have weakened the hypothesis - Becker assumes $R$ contains a field; we have correspondingly weakened the conclusion - we do not assert the Hironaka remainder $r$ is unique.
In fact, the remainder will not be unique in the mixed characteristic case.
However, with a single adjustment, Becker's proof goes through exactly as it does in the equal characteristic case.
In the proof of Proposition 2.1, Becker inductively constructs the remainder $r$ and we must simultaneously construct a power series representation for $r$ as $T(r)$ depends upon the choice of representation.
Intuitively, as we build $r$ as a limit, we adjust the representation as little as possible so that we also have a limit of the representations.
To be precise, in Becker's notation, for all $\alpha>\beta$, we set $c(t_\beta,r_\alpha)=c(t_\beta,r_\beta)$ and so we get a representation with $c(t_\beta,r)=c(t_\beta,r_\beta)$.

\begin{defn}
Let $y_1,\dots,y_m$ be a standard basis for an ideal $I$.
A basis element $y_i$ is called superfluous if $y_i\in \m I+(y_{i+1},\dots,y_m)R$.
\end{defn}

\begin{lemma}
Let $y_1,\dots,y_m$ be a standard basis for an ideal $I$.
The non-superfluous elements in the standard basis comprise a minimal generating set for $I$ and so $\mu(I)=m-\ell$ where $\ell$ is the number of superfluous elements.
\end{lemma}

\begin{proof}
It suffices to consider the vector space $V=I/\m I$ in place of $I$.
Obviously $\bar{y}_1,\dots,\bar{y}_m$ span $V$.
Starting from the left, if $y_j$ is superfluous, we have that $\bar{y_j}$ is in the span of $\{\bar{y}_q\mid q>j\}$ and so we still have a spanning set with $\bar{y}_j$ removed.
Moreover, as $y_j$ was the first superfluous element, the removal does not affect the superfluousness of the other elements.
Thus in turn we may remove all superfluous elements and still have a generating set.
Finally, to see that the remaining elements are linearly independent, we consider the equation $\sum c_i\bar{y}_i=0$ and note that if $c_j$ is the leftmost nonzero coefficient, $y_j$ must be superfluous, a contradiction.
\end{proof}

\begin{prop}\label{1.4}
Let $J$ be a perfect ideal of grade $g$.
Suppose $x_1,\dots,x_d$ is a system of parameters such that $x_{g+1},\dots,x_d$ comprise a system of parameters for $R/J$.
Let $f_1,\dots,f_m$ be a standard basis for $J$.
Then the leading terms of the standard basis elements are monomials of the form $x_1^{e_1}\dots x_g^{e_g}$.
Moreover, if $x_1^{e_1}\dots x_g^{e_g}$, $x_1^{c_1}\dots x_g^{c_g}$ are two different leading terms, $e_i\neq c_i$ for some $i<g$.
\end{prop}

\begin{proof}
First note that because $J$ is perfect, the depth of $R/J$ must be $d-g$ and so $x_{g+1},\dots,x_d$ is a regular sequence on $R/J$.
Now, if the first conclusion of the proposition were false, there must be an element $f_i$ in the standard basis with $f_i\in (x_{g+1},\dots,x_d)R$.
Choose $c>g$ minimal such that $f_i\in (x_c,\dots,x_d)R$.
So $LT(f_i)\notin (x_{c+1},\dots,x_d)R$.
Now choose $j$ minimal so that $f_i\in (x_c,\dots,x_j)R$.
If $j=c$, $f_i=x_ch$ for some $h\in R$ and as $x_c$ is regular on $R/J$, $h\in J$.
This contradicts the notion that $f_i$ is part of a standard basis.
So $j>c$.
Write $f_i=s+x_jt$ with $s\in (x_c,\dots,x_{j-1})R$.
As $x_c,\dots, x_j$ is a regular sequence on $R/J$, we may write $t=s'+t'$ with $s'\in (x_c,\dots,x_{j-1})R$ and $t'\in J$.
Let $\tilde{f}_i=f_i-x_jt'=s+x_js'$.
Then $\tilde{f}_i\in J$, $LT(\tilde{f}_i)=LT(f_i)$, and so we may replace $f_i$ by $\tilde{f}_i$ in our standard basis.
But $\tilde{f}_i\in (x_c,\dots,x_{j-1})R$ and by induction, we obtain a contradiction.

The moreover statement is obvious.  
If $e_i=c_i$ for all $i<g$, then one of the monomials must divide the other and so the two corresponding elements cannot be members of the same standard basis.
\end{proof}

\begin{prop}\label{Lemma 487}
Let $J$ be a perfect ideal of grade $g$.
Let $x_1,\dots,x_d$ be a regular system of parameters such that $x_{g+1},\dots,x_d$ comprise a system of parameters for $R/J$.
Suppose $(x_1,\dots,x_{g-1})^\ell R\subset \m J+\m^{\ell+1}+(x_g,\dots,x_d)R$.
Then $\mu(J)\leq \begin{pmatrix} g+\ell-2\\ g-1 \end{pmatrix}$.
\end{prop}

\begin{proof}
First we improve the hypothesis.
Since $x_1,\dots,x_d$ generate $\m$, the hypothesis implies $\m^\ell\subset \m J+\m^{\ell+1}+(x_g,\dots,x_d)R$.
Then, by Nakayama's Lemma, $(x_1,\dots,x_{g-1})^\ell\subset\m^\ell\subset\m J+(x_g,\dots,x_d)R$.

Using our parameters $x_1,\dots,x_d$, we deviate from the rest of this article and order monomials in the usual way.
To wit, we order the monomials $x_1^{e_1}\dots x_d^{e_d}$ using lexicographic order on the $d$-tuples $(e_d,\dots,e_1)$.
We choose a standard basis $f_1,\dots, f_m$ for $J$, numbered so that $LT(f_i)<LT(f_j)$ whenever $i<j$.
 Let $f=f_i$ be an element of that basis.
By the previous proposition, $LT(f)=x_1^{e_1}\dots x_g^{e_g}$.
Let $q=\sum_{i=1}^{g-1}e_i$.
We claim that if $q\geq \ell$, $f$ is a superfluous generator.
Now note that $q\geq \ell$ implies that $LT(f)\in (x_1,\dots,x_{g-1})^\ell R$.
By the nature of our ordering, this tells us that every term in the Hironaka representation of $f$ - and so $f$ itself - is in $x_g^{e_g}((x_1,\dots,x_{g-1})R^\ell + (x_g,\dots,x_d))R$.
Thus $f\in\m J+(x_g^{e_g+1},\dots,x_d)R$ by the hypothesis.
Hence, modulo $\m J$, $f\equiv s_1$ with $s_1\in (x_g^{e_g+1},\dots,x_d)R\cap J$ and so in particular $LT(f)<LT(s_1)$.

Since modulo $(x_{g+1},\dots,x_d)R$, $\bar{J}\cap \m^h\subset \m \bar{J}$ for sufficiently large $h$, the set of leading terms of elements in $J-\m J$ is finite and so we can find maximal elements.
Necessarily there exists $f_j$ such that $LT(f_j)\mid LT(s_1)$.
We can then write $s_1=a_1f_j+s_2$ with $LT(s_1)<LT(s_2)$.
If $a_1$ is not a unit, $s_2$ is congruent to $s_1$ modulo $\m J$.  
If $a_1$ is a unit, $LT(f_j)>LT(f)$ and so $j>i$.
We repeat the process with $s_2$, noting $s_2=a_2f_\ell+s_3$ and either $\ell>i$ or $a_2f_\ell\in \m J$.
Since we cannot find infinitely many $s_j\in J- \m J$, for sufficiently large $N$, we must have $s_N\in \m J$.
Thus $f=f_i\in \m J+(f_{i+1},\dots,f_m)R$ and $f_i$ is a superfluous generator and the claim is shown.

To obtain an upper bound on $\mu(J)$, it suffices to obtain an upper bound on the number of generators which are not superfluous.
There is at most one generator for any $(g-1)$-tuple $(e_1,\dots, e_{g-1})$ and that generator is superfluous unless $\sum_{i=1}^{g-1}e_i<\ell$.
Standard combinatorics now tell us that the number of non-superfluous generators is at most $ \begin{pmatrix} g+\ell-2\\ g-1 \end{pmatrix}$.
\end{proof}

\section{the main results}

We first establish Conjecture \ref{E1} in the case where $g=\dim R$.

\begin{lemma}\label{socle}
Let $(R,\m,k)$ be a regular local ring and $I$ an ideal of $R$.
Then $\lambda (I\cap \m^{n-1}+\m^n/\m^n)\geq \mu(I+\m^n/\m^n)$.
\end{lemma}

\begin{proof}
We shall prove the lemma under the assumption that $I\subseteq \m^{n-i}$ for each $i\leq n$ by induction on $i$ and the $i=n$ case is the full lemma.
For $i=1$, we have $\mu(I+\m^n/\m^n)=\mu(I\cap\m^{n-1}+\m^n/\m^n)=\lambda (I\cap \m^{n-1}+\m^n/\m^n)$ and the conclusion holds.
For $i>1$, let $\ell=\mu(I+\m^{n-1}/\m^{n-1})$ and let $I'\subseteq I$ be an $\ell$-generated ideal such that $I'+\m^{n-1}=I+\m^{n-1}$.
Next choose an ideal $I''\subseteq I\cap\m^{n-1}$ minimal with respect to the property that $I'+I''=I$.
We now have a direct sum $I+\m^n/\m^n=I'+\m^n/\m^n\oplus I''+\m^n/\m^n$.
Next choose $t\in\m-\m^2$.
We have $\ell=\mu(tI'+\m^n/\m^n)$.
As $tI'\subset  \m^{n-(i-1)}$ and $I''\subset \m^{n-1}$, the induction assumption gives

$$   \mu(I+\m^n/\m^n)=\mu(I'+\m^n/\m^n)+\mu(I''+\m^n/\m^n)   $$
$$=\mu(tI'+\m^n/\m^n)+\mu(I''+\m^n/\m^n)$$
$$\leq \lambda(tI'\cap\m^{n-1}+\m^n/\m^n)+\lambda(I''\cap\m^{n-1}+\m^n/\m^n)$$
$$=\lambda((tI'+I'')\cap\m^{n-1}+\m^n/\m^n)\leq\lambda(I\cap\m^{n-1}+\m^n/\m^n)     $$
\end{proof}

\begin{prop}\label{g=d}
Conjecture \ref{E1} is true when $g=\dim R$.
\end{prop}

\begin{proof}
It is well known that a simple calculation gives $\lambda(\m^{n-1}/\m^n)=\begin{pmatrix} g+n-2\\ g-1 \end{pmatrix}$ and the result then follows immediately from Lemma~\ref{socle}.
\end{proof}

Next we present an example which shows that neither conjecture is true as stated in the introduction.
In fact, without additional hypotheses, even in the case $g=n=3$, there is no bound on the number of generators which is independent of the dimension of the ring.
The key feature of the example is that $I\subset x_1R+\m^n$.

\begin{eg}
Let $K$ be a field and let $x_1,x_2,x_3,y_0,y_1,\dots,y_N$ be indeterminates with $N\geq 3$.
Let $R=K[X,Y]$ be the the $(N+4)$-dimensional polynomial ring with $\m$ the maximal ideal generated by the indeterminates.
Suppose $I=(x_1^2,x_1x_2,x_1x_3,x_1y_0+x_2^N,x_1y_1+x_2^{N-1}x_3,\dots,x_1y_N+x_3^N)R$.
Then $I$ is a height three perfect ideal such that $\mu(I)=\mu(I+\m^3/\m^3)=N+4$.
\end{eg}

\begin{proof}
Obviously $\mu(I+\m^3/\m^3)\leq \mu(I)\leq N+4$.
However $I+\m^3/\m^3=x_1\m/\m^3\cong \m/\m^2$ and so $\mu(I+\m^3/\m^3)=\mu(\m)=N+4$.
This proves the second conclusion.
As $x_1^2,x_1y_0+x_2^N,x_1y_N+x_3^N\in I$, it is clear that $I\subseteq (x_1,x_2,x_3)R=\sqrt{I}$ and so $I$ has height three.
As $R$ is regular, $R/I$ must have finite projective dimension and so the proof will be complete if we can show that $y_0,\dots,y_N$ is a regular sequence on $R/I$.

Suppose $y_0r\in I$, that is, $$y_0r=a_1x_1^2+a_2x_1x_2+a_3x_1x_3+\sum_{j=0}^Nb_j(x_1y_j+x_2^{N-j}x_3^j).$$
If all $b_j$ are zero, the fact that $y_0,x_1,x_2,x_3$ is a regular sequence in $R$ forces $r\in(x_1^2,x_1x_2,x_1x_3)R\subset I$.
Otherwise, we have some $\ell\geq 0$ chosen to be minimal over all expressions for $s\in r+I$ such that $b_\ell\neq 0$ and
$$y_0s=a_1x_1^2+a_2x_1x_2+a_3x_1x_3+\sum_{j=0}^\ell b_j(x_1y_j+x_2^{N-j}x_3^j).$$
Again using the regularity of $y_0,x_1,x_2,x_3$, $b_\ell x_2^{N-\ell}x_3^\ell\in (y_0,x_1,x_2^{N-\ell+1})R$ forces
$b_\ell\in (y_0,x_1,x_2)R$.
If $\ell=0$, we actually get $b_\ell\in (y_0,x_1)R$.
We write $b_\ell=y_0c_0+x_1c_1+x_2c_2$.
Since $x_1(x_1y_\ell+x_2^{N-\ell}x_3^\ell)=y_\ell x_1^2+x_1x_2^{N-\ell}x_3^\ell\in (x_1^2,x_1x_2,x_1x_3)R$, we may reduce to the case $c_1=0$.
For $\ell>0$, $x_2(x_1y_\ell+x_2^{N-\ell}x_3^\ell)=y_\ell(x_1x_2)+x_3(x_1y_{\ell-1}+x_2^{N-\ell+1}x_3^{\ell-1})-y_{\ell-1}x_1x_3$ and so we also can reduce to the case $c_2=0$.
(We already had $c_2=0$ in the $\ell=0$ case.)
Finally, we may replace $s$ by $s'=s-c_0(x_1y_\ell+x_2^{N-\ell}y_3^\ell)\in r+I$ and so we can reduce to the case $b_\ell=0$, the desired contradiction which shows $y_0$ is regular on $R/I$.

Now suppose $i>0$ and $y_ir\in I+(y_0,\dots,y_{i-1})R$.
We proceed as above, working in the ring $K[x_1,x_2,x_3,y_i,\dots,y_N]$.
So we have
$$y_ir=a_1x_1^2+a_2x_1x_2+a_3x_1x_3+\sum_{j=0}^{i-1}b_j(x_2^{N-j}x_3^j)+\sum_{j=i}^Nb_j(x_1y_j+x_2^{N-j}x_3^j).$$
If $b_j=0$ for all $j\geq i$, the fact that $y_i,x_1,x_2,x_3$ is a regular sequence forces $r\in(x_1^2,x_1x_2,x_1x_3,x_2^N,\dots,x_2^{N-i+1}x_3^{i-1})R\subset I$.
Otherwise, we have some $\ell\geq i$ chosen to be minimal over all expressions for $s$ for $s\in r+I$ such that $b_\ell\neq 0$ and
$$y_is=a_1x_1^2+a_2x_1x_2+a_3x_1x_3+\sum_{j=0}^{i-1}b_j(x_2^{N-j}x_3^j)+\sum_{j=i}^\ell b_j(x_1y_j+x_2^{N-j}x_3^j).$$
Again using the regularity of $y_i,x_1,x_2,x_3$, $b_\ell x_2^{N-\ell}x_3^\ell\in (y_0,x_1,x_2^{N-\ell+1})R$ forces
$b_\ell\in (y_0,x_1,x_2)R$.
We write $b_\ell=y_0c_0+x_1c_1+x_2c_2$.
Since $x_1(x_1y_\ell+x_2^{N-\ell}x_3^\ell)=y_\ell x_1^2+x_1x_2^{N-\ell}x_3^\ell\in (x_1^2,x_1x_2,x_1x_3)R$, we may reduce to the case $c_1=0$.
For $\ell>i$, $x_2(x_1y_\ell+x_2^{N-\ell}x_3^\ell)=y_\ell(x_1x_2)+x_3(x_1y_{\ell-1}+x_2^{N-\ell+1}x_3^{\ell-1})-y_{\ell-1}x_1x_3$.
For $\ell=i$, $x_2(x_1y_\ell+x_2^{N-\ell}x_3^\ell)=y_\ell(x_1x_2)+x_3(x_2^{N-\ell+1}x_3^{\ell-1})$.
Thus we can also reduce to the case $c_2=0$.
Finally, we may replace $s$ by $s'=s-c_0(x_1y_\ell+x_2^{N-\ell}y_3^\ell)\in r+I$ and so we can reduce to the case $b_\ell=0$, the desired contradiction which completes the proof.
\end{proof}

As noted above, the key element of the example is the existence of an element $y\in R$ such that $I\subseteq yR+\m^n$.
For $g=4$, we can modify this example by adding $x_4^2$ to the list of generators for $I$.
Here we have $y_1,y_2\in R$ with $I\subseteq (y_1,y_2)+\m^n$.
And so on for larger $g$.
To circumvent these examples, we adjust the conjectures.

\begin{conj}\label{C1}
 Let $(R,\m)$ be a local ring and suppose $J$ is a perfect ideal of grade $g$.
 Further suppose there does not exist a height $g-2$ ideal $I$ such that $J\subseteq I+\m^n$.
Then $\mu(J+\m^n/\m^n)\leq \begin{pmatrix} g+n-2\\ g-1 \end{pmatrix}$.
\end{conj}

\begin{conj}\label{new}
Let $(R,\m)$ be a regular local ring and suppose $J$ is a perfect ideal of grade $g$.
 Further suppose there does not exist a height $g-2$ ideal $I$ such that $J\subseteq I+\m^n$.
If $\mu(J+\m^n/\m^n)\geq \begin{pmatrix} g+n-3\\ g-2 \end{pmatrix}$, then $\mu(J)\leq \begin{pmatrix} g+n-2\\ g-1 \end{pmatrix}$.
\end{conj}

Next we develop a theorem which is weaker than Conjecture~\ref{new}, but we can actually prove at this time.
While this bound is generally larger than that proposed above, it actually coincides with the bound given by Hilbert-Burch when $g=2$.

\begin{defn}
Let $(R,\m,k)$ be a complete regular local ring.
Fix a regular system of parameters $x_1,\dots,x_d$ for $R$.
There is a natural additive map $\Phi_n:\m^{n-1}\to k[x_1,\dots,x_d]$ whose image consists of all homogeneous elements of degree $n-1$.
For $f\in \m^{n-1}$, let $f=\sum c(t_\alpha,f)t_\alpha$ be a Hironaka representation.
Set $\Phi_n(f)=\sum_{\deg t_\alpha=n-1}\overline{c(t_\alpha,f)}t_\alpha$ where $\overline{c(t_\alpha,f)}\in k$.
\end{defn}

Note this function is well defined in the mixed characteristic case.
Also, while $\Phi_n$ depends on the choice of parameters, changing parameters merely adjusts the function by a graded isomorphism of $k[x_1,\dots,x_d]$.

\begin{lemma}\label{combinatorics}
Let $(R,\m,k)$ be a complete regular local ring of dimension $d>0$ and let $I$ be an ideal of $R$.
Suppose $\Phi_n(I\cap \m^{n-1})$ generates an ideal of height $d$ in $k[x_1,\dots,x_d]$.
Then $\m^{(n-2)d+1}\subset I$.
Further, if $d>1$ and $n>2$, we have $\m^{(n-2)d+1}\subset \m I$.
\end{lemma}

\begin{proof}
It suffices to prove the result with $I$ replaced by $I\cap\m^{n-1}$.
Then certainly there is a $d$-generated subideal $I'$ such that $\Phi_n(I')$ generates an ideal of height $d$ in $k[x_1,\dots,x_d]$.
We can then harmlessly replace $I$ by $I'$ and so assume $I=(y_1,\dots, y_d)R$ and that $z_1=\Phi_n(y_1),\dots,z_d=\Phi_n(y_d)$ are homogeneous elements of $k[x_1,\dots,x_d]$ of degree $n-1$ which constitute a regular sequence.
We let $M_i$ denote the homogeneous summand of degree $i$ in $k[x_1,\dots,x_d]$.
This is a $k$-vector space, and by simple combinatorics, we see that $\dim M_i= \begin{pmatrix} i+d-1\\ d-1 \end{pmatrix}$.
In our setup, $z_1,\dots,z_d\in M_{n-1}$.

To see that $\m^{(n-2)d+1}\subset I$, by Nakayama's Lemma it suffices to show that $\m^{(n-2)d+1}\subset I+\m^{(n-2)d+2}$.
Certainly $I$ contains all elements of the form $y_1a_1+\dots+y_da_d$ with each $a_i\in \m^{(n-2)(d-1)}$.
Letting $b_i=\Phi_{(n-2)(d-1)+1}(a_i)$, it is easy to see that $\Phi_{(n-2)d+2}(y_1a_1+\dots+y_da_d)=z_1b_1+\dots+z_db_d$, an element of  $M_{(n-2)d+1}$.
It remains to show that these elements constitute all of $M_{(n-2)d+1}$, which can be accomplished by comparing the dimensions of the entire space and the subspace.

As $(n-2)d+1+ (d-1)=(n-1)d$, we have already seen that $\dim M_{(n-2)d+1}= \begin{pmatrix} (n-1)d\\ d-1 \end{pmatrix}$.
For the subspace, we let $V_{ih}$ be the vector space $(z_1,\dots,z_i)M_h$.
(The space is zero dimensional for $h<0$.)
Certainly $\dim V_{1h}=\dim M_h$.
Further $\dim V_{i+1,h}-\dim V_{ih}=\dim M_h-\dim (V_{ih}\cap z_{i+1}M_h)$.
Now $z_{i+1}b_{i+1}\in V_{ih}\iff b_{i+1}\in (z_1,\dots,z_i)M_{h-n+1}$.
Thus $\dim V_{i+1,h}=\dim V_{ih}+\dim M_h-\dim M_{i,h-n+1}$.
The proof is complete if we can show that $\dim V_{d,(n-2)(d-1)}=\dim M_{(n-2)d+1}$.

This can be done by a messy calculation, but there is an easier way to see this.
Notice that the numbers in our calculation are independent of the elements $z_i$.
Hence we can check our equality in the case where $z_1=x_1^{n-1},\dots,z_d=x_d^{n-1}$.
However, it is readily apparent that $\m^{(n-2)d+1}=(z_1,\dots,z_d)M_{(n-2)(d-1)}$ in the special case and so the first inclusion holds in general.

To see the moreover statement, note that, after our reductions, $\mu(I)=\mu(I+\m^n/\m^n)=d$ and so every element of $\m^{(n-2)d+1}$ must be in $\m I$ unless $(n-2)d+1=n-1$.
This only happens when $d=1$ or $n=2$.
\end{proof}

\begin{ques}
If $R$ is a regular local ring of dimension $d$ and $I_1,I_2$ are minimal reductions of $\m^n$, we see from this lemma that both ideals contain $\m^{(n-1)d+1}$.
It is also clear using the same reasoning that neither will contain $\m^{(n-1)d}$.
Thus $I_i$ does not contain $\overline {I_i^{d-1}}$ but does contain $\overline{I_i^{d-1+\frac 1d}}$.
This suggests the following question.
If $I_1,I_2$ are minimal reductions of the same integral closure and $a$ is a rational number, is it true that $I_1$ contains $\overline{I_1^a}$ if and only if $I_2$ contains  $\overline{I_2^a}$.
Of course, $\overline{I_1^a}=\overline{I_2^a}$.
\end{ques}

\begin{thm}\label{main}
 Let $(R,\m,k)$ be a complete regular local ring of dimension $d$ with $|k|=\infty$ and suppose $J$ is a perfect ideal of grade $g\geq 2$.
 Let $x_1,\dots,x_d$ be a regular system of parameters for $R$.
 Further suppose $n>2$ and $\Phi_n(J\cap \m^{n-1})$ generates an ideal of $k[x_1,\dots,x_d]$ of height at least $g-1$.
Then $\mu(J)\leq \begin{pmatrix} (g-1)(n-1)+\alpha \\ g-1 \end{pmatrix}$ where $\alpha=1$ if $g=2$ and zero otherwise.
\end{thm}

\begin{proof}
First we ``improve" our system of parameters.
It is a relatively simple exercise in prime avoidance to choose a regular system of parameters $y_1,\dots,y_d$ such that, in $S=k[x_1,\dots,x_d]$, $\Phi_n(J\cap{\m}^{n-1})S+(\Phi_2(y_g),\dots,\Phi_2(y_d))S$ has height $d$.
To do this, as $\Phi_n(J\cap \m^{n-1})$ has height at least $g-1$, we can easily choose $\bar{y}_g,\dots, \bar{y}_d\in S$ so that $\Phi_n(J\cap{\m}^{n-1})S+(\bar{y}_g,\dots,\bar{y}_d)S$ has height $d$.
This is simple prime avoidance.
Then we lift each $\bar{y}_i$ to an element $y_i'\in R$.
For $i<g$, let $y_i=x_i$.
For $i\geq g$, we let $y_i=y_i'+a_i$ for some well chosen $a_i\in \m^2$; $y_i$ will also be a lifting of $\bar{y}_i$.
Since $\m^2$ will not be contained in any primes we are trying to avoid, we can easily choose the $a_i'$s so that $y_1,\dots,y_d$ is a regular system of parameters.
As changing the regular parameters $x_1,\dots,x_d$ to $y_1,\dots,y_d$ just gives a graded isomorphism of $k[x_1,\dots, x_d]$, we can assume $x_j=y_j$ for all $j$ and our hypothesis remains valid.

Now let $\tilde{R}=R/(x_g,\dots,x_d)$ and $\tilde{J}=J+(x_g,\dots,x_d)R$.
Then $\Phi_n(\tilde{J}\cap\tilde{\m}^{n-1})$ generates an ideal of height $g-1$ in $k[x_1,\dots,x_{g-1}]$.
Next we may apply Lemma~\ref{combinatorics} to $\tilde{R}$  and see that
$(x_1,\dots,x_{g-1})^{(n-2)(g-1)+1}R \subset J+(x_g,\dots,x_d)R$.
If $g-1>1$, we in fact get $(x_1,\dots,x_{g-1})^{(n-2)(g-1)+1}R\subset \m J+(x_g,\dots,x_d)R$, but when $g=2$, we must add $1$ to the exponent to get the ideal inside $ \m J+(x_g,\dots,x_d)R$.
Finally we apply Proposition~\ref{Lemma 487} to $R$ with $\ell=(n-2)(g-1)+1+\alpha$ to obtain the desired conclusion.
\end{proof}

\begin{remark}
In the statement of the conjectures, it is assumed that there does not exist an ideal $I$ of height $g-2$ such that $J\subseteq I+\m^n$.
In the statement of the main theorem, it is assumed that $\Phi_n(J)$ generates an ideal of height at least $g-1$.
It is not hard to see that these assumptions coincide for $g=3$ and of course for $g=2$.
However the equivalence is not readily apparent for $g>3$ though it seems reasonable.
\end{remark}

\section{Exploring Conjecture~\ref{new} when $g=3$}

In this section, we seek to obtain a few partial results toward proving Conjecture~\ref{new} when $g=3$.
Accordingly, we will organize a standard set of assumptions for a fixed value of $n$.

\begin{notation}\label {notation}
Let $(R,\m,k)$ be a complete local ring of dimension $d\geq 3$ with an infinite residue field.
Let $J\subseteq \m^{n-1}$ be a perfect ideal of $R$ of height $g=3$.
Assume $\mu(J+\m^n/\m^n)\geq n$ and that $J\nsubseteq yR+\m^n$ for any element $y\in\m$. 
\end{notation}

Next we introduce one of our basic techniques.
We will want our regular system of parameters to work well with our function $\Phi_n$ and so we make the following definition.

\begin{defn}
 We call a regular system of parameters $x_1,\dots,x_d$ $n$-compliant provided it satisfies the following two conditions:
 \begin{enumerate}
  \item Any subset of $d-3$ elements is a system of parameters for $R/J$.
Alternately, as the depth of $R/J$ is $d-3$, any such subset is a regular sequence on $R/J$.
  \item The height of $\Phi_n(J)k[x_1,\dots,x_d]+(z_1,\dots,z_\ell)k[x_1,\dots, x_d]$ is at least $\ell+2$ for every subset $\{z_1,\dots, z_\ell\}\subset \{x_1,\dots,x_d\}$ with $\ell\leq d-2$.
\end{enumerate}
\end{defn}

Finding an $n$-compliant (regular) system of parameters is a simple exercise in prime avoidance.
Further, if we already have an $n$-compliant system, a new system obtained by replacing $x_i$ by $x_i'=x_i+\alpha x_j$ will be $n$-compliant for all but finitely many choices of $\alpha$ modulo $\m$.
To see this, note that we simply need $x_i'$ to avoid a finite set of primes which $x_i$ already avoids.
Consider any such prime $P$.
If $x_i+\alpha_1 x_j, x_i+\alpha_2 x_j\in P$ with $\alpha_1-\alpha_2\notin \m$, then $(\alpha_1-\alpha_2)x_j\in P$ and so $x_j\in P$.
However, in this case, $x_i'\notin P$ for all $\alpha$.
Throughout our proofs, we will want to ``improve" our system of parameters by replacing $x_i$ by $x_i'$.
We will need $\alpha$ to not be a solution of any of a finite set of polynomial equations over $k$.
To do this, $\alpha$ must avoid a finite set.
Insisting that the new set of parameters is $n$-compliant merely increases the finite set of invalid choices and poses no difficulty as $k$ is infinite.

\begin{lemma}\label{dimension reduction}
Fix $n\geq 3$ and assume the setting of Notation~\ref{notation}.
Assume one of the following holds.
\begin{enumerate}
  \item $d>n$
  \item $d=n$ and $\mu(J+\m^n/\m^n)> n$
  \item $d= n$ and $\{r^{n-1}\mid r\in\m\}\nsubseteq J+\m^n$
  \end{enumerate}
Then there exists $z\in\m-\m^2$ such that $\mu ((J+\m^n+zR)/(\m^n+zR))\geq n$ and $J\nsubseteq \m^n+(y,z)R$ for any $y\in\m$.
Moreover, unless $d=3$, we may choose $z$ to be a parameter in $R/J$.
\end{lemma}

\begin{remark}
If $n-1$ is a power of the characteristic of $k$, $\{r^{n-1}\mid r\in\m\}\subseteq J+\m^n$ if and only if $J+\m^n$ contains $x_1^{n-1},\dots, x_d^{n-1}$ for some system of parameters $x_1,\dots,x_d$.
On the other hand, if $n-1$ is not a power of the characteristic of $k$, when $d=n$, either condition (2) or condition (3) always holds.
\end{remark}

We now prove the lemma.

\begin{proof}
Let $x_1,\dots, x_d$ be an $n$-compliant system of parameters for $R$.
We let $M_i$ denote the vector space of homogeneous forms of degree $i$ in $k[x_1,\dots,x_d]$.
We let $V$ denote the subspace of $M_{n-1}$ which is spanned by $\Phi_n(J)$ and set $m=\dim V$.
Assume we have a counterexample where $d\geq n$ and no such $z$ exists.
Necessarily there exists $f_1\in J$ such that $\phi_1=\Phi_n(f_1)=x_1h_1$ for some $h_1\in M_{n-2}$.
More precisely, there exists $f_1\in J$ such that $\phi_1=\Phi_n(f_1)=x_1^qw_1$ for some $w_1\in M_{n-q-1}$ such that $w_1\notin x_1M_{n-q-2}$.
We can choose our $n$-compliant system of parameters  and $f_1$ to minimize $q$.
There exists a monomial $x_1^{e_1}x_2^{e_2}\dots x_d^{e_d}\in S(\phi_1)$.
For most changes of parameters $x_i'=x_i+\alpha_ix_1$ for all $i>1$, we can force $x_1^{n-1}\in S(\phi_1)$.
Then we note that $J=f_1R+J_2$ with $x_1^{n-1}\notin S(\phi)$ for all $\phi\in \Phi(J_2)$.
Next there exists $f_2\in J$ such that $\phi_2=\Phi_n(f_2)=x_2h_2$ for some $h_2\in M_{n-2}$.
Clearly $f_2\in J_2$.
There exists a monomial $x_1^{e_1}x_2^{e_2}\dots x_d^{e_d}\in S(\phi_2)$.
For most changes of parameters $x_i'=x_i+\alpha_ix_1$ for all $i>2$, we can force $x_1^{e_1}x_2^{n-1-e_1}\in S(\phi_2)$.
Then $J=(f_1,f_2)R+J_3$ with $x_1^{n-1},x_1^{e_1}x_2^{n-1-e_1}\notin S(\phi)$ for all $\phi\in \Phi(J_3)$.
Continuing, there exists $f_3\in J$ such that $\phi_3=\Phi_n(f_3)=x_3h_3$ for some $h_3\in M_{n-2}$ and as before, we have $f_3\in J_3$.
After adjusting the support of $\phi_3$, $J_4$ will be the subideal of $J_3$ with the obvious restriction on the elements of the support of $\Phi_n(J_4)$.
And so on.
It is clear that $\phi_1,\dots,\phi_n$ are linearly independent and remain linearly independent modulo $x_{n+1}$.
So there cannot be an $x_{n+1}$ and we must have $d=n$.

Next we re-examine the situation before the final step.  
We had $J=(f_1,\dots,f_{n-1})R+J_n$.
As $\phi_1,\dots,\phi_{n-1}$ are linearly independent modulo the span of $\Phi_n(J_n)$ and we are assuming $\mu ((J+\m^n+x_nR)/(\m^n+x_nR))< n$, we must have $\Phi_n(J_n)\subseteq x_nM_{n-2}$.
Now, because we have an $n$-compliant system of parameters, $J\nsubseteq \m^{n} + (z_1,\dots,z_{d-1})R$ for any proper subset of our system of parameters.  
It follows that, for every $i$, $x_i^{n-1}$ is in the support of some $\phi_i$ and trivially we see that it must be that $x_i^{n-1}\in S(\phi_i)$ for $i<d$.
We can also choose $f_d,\dots,f_m$ so that $x_d^{n-1}\in S(\phi_d)$.
Write $J=(f_1,\dots,f_d)R+J_{n+1}$ with $x_i^{n-1}\notin S(\phi)$ for all $i$ and $\phi\in \Phi(J_{n+1})$.
For any $i\leq d$, $\mu ((J+\m^n+x_iR)/(\m^n+x_iR)< n$ forces $\Phi_n(J_{n+1})\subseteq x_iM_{n-2}$.
But then the homogeneous degree $n-1$ form $\phi_{d+1}$, if it exists, must be divisible by $\prod_{i=1}^dx_i$, an impossibility.
It follows that $m=d=n$.

Again, by our assumption that the conclusion of the lemma is false and the minimality assumption on $q$, we have, for all but finitely many values of $\gamma$ that there exists a nonzero element $(x_1+\gamma x_2)^qy\in V$.
We may write $(x_1+\gamma x_2)^qy=\sum_{i=1}^na_i\phi_i$ with each $a_i\in k$.
Simple consideration of the support yields $a_i=0$ whenever $i>2$ and so we have 
$(x_1+\gamma x_2)^qy=a_1\phi_1+a_2\phi_2=a_1x_1^qw_1+a_2x_2^qw_2$.
First consider the case where $q=1$ or $q$ is a power of the characteristic of $k$.
Here the left hand side of the equation simplifies to $x_1^q+\gamma^qx_2^q$.
Regrouping, $x_1^q(y-a_1w_1)=x_2^q(a_2w_2-\gamma^q y)$.
This means that either $y-a_1w_1=a_2w_2-\gamma^q y=0$, in which case we may assume $w_1=w_2$ or there exists a nonzero polynomial $t$ such that $y-a_1w_1=x_2^qt$ and $a_2w_2-\gamma^q y=x_1^qt$.
Eliminating $y$ from these equations gives $a_2w_2-\gamma^q(x_2^qt+a_1w_1)=x_1^qt$ and so 
$(x_1^q+\gamma^q x_2^q)t=a_2w_2-a_1\gamma^q w_1$.
Choosing another $\delta$ so that $(x_1^q+\delta^q x_2^q)$ does not divide $t$, we obtain $(x_1^q+\delta ^qx_2^q)t'=a_2'w_2-a_1'\delta^q w_1$.
As $(x_1^q+\delta^q x_2^q)$ does not divide $(x_1^q+\gamma^q x_2^q)t$, $a_2'w_2-a_1'\delta^q w_1$ and $a_2w_2-a_1\gamma^q w_1$ are not multiples of each other and so
 $a_2a_1'\delta^q-a_1a_2'\gamma^q\neq 0$.
 It follows that $w_1,w_2$ are linear combinations of $(x_1^q+\gamma^q x_2^q)t$ and $(x_1^q+\delta^q x_2^q)t'$ and so $w_1,w_2\in (x_1^q,x_2^q)$.
Thus either $w_1=w_2$ or $w_1,w_2\in (x_1^q,x_2^q)$.
Symmetrically, for any $i>2$, either $w_1=w_i$ or $w_1,w_i\in (x_1^q,x_i^q)$.
In fact, for any $i\neq j$, either $w_i=w_j$ or $w_i,w_j\in (x_i^q,x_j^q)$.
Suppose $w_1\neq w_i$ for all $i>1$.
Then $w_1\in \cap_{i=2}^n(x_1,x_i)\cap M_{n-1-q}\subset (x_1)$.
This is a contradiction.
So $w_1=w_i$ for some $i\geq 2$.
Suppose there also exists $j>1$ such that $w_1\neq w_j$.
Then $w_1=w_i\in (x_i^q,x_j^q)$, which is impossible since $x_i^{n-1}\in S(\phi_1)$.
So we must have $w_i=w_1$ for all $i$.
Then every element of $V$ is a multiple of $w_1$, which can only be true if $w_1$ is a unit and again $q=n-1$.
Then, modulo $\m^n$, $\phi_i\equiv x_i^{n-1}$ and this is not allowed by the second assumption.

Now we assume $q>1$ and $q$ is not a power of the characteristic of $k$.
Here $(x_1+\gamma x_2)^q$ has terms which are not in $(x_1^q,x_2^q)$ and so $(x_1+\gamma x_2)^qy=a_1x_1^qw_1+a_2x_2^qw_2$ forces $y\in (x_1,x_2)$ and so $w_1,w_2\in (x_1,x_2)$.
Symmetrically, $w_1\in (x_1,x_i)$ for every $i$, which forces $w_1\in (x_1)$, violating the definition of $q$.
This completes the proof.

\end{proof}

\begin{remark}\label{remark}
This lemma will allow us to reduce the proof of Conjecture~\ref{new} when $g=3$ and $n$ is fixed to the case $d\leq n$.
Choose our parameters so that $x_d$ is the $z$ given by the lemma.
Let $\bar{R}=R/x_dR$ and let $\bar{J}=J+x_dR/x_dR$.
Then $\mu(\bar{J}+\m^n/\m^n)\geq n$.
Since the height of $(\phi_1,\dots,\phi_m,x_d)k[x_1,\dots, x_d]$ is at least $3$, $(\phi_1,\dots,\phi_m)k[x_1,\dots, x_d]\nsubseteq (x_d,y)k[x_1,\dots, x_d]$ for any $y\in\m$ and so $\bar{J}\nsubseteq \bar{\m}^n+y\bar{R}$.
Since $x_d$ is a parameter in the Cohen-Macaulay ring $R/J$,
it quickly follows that $(J+x_dR)/x_dR$ is a perfect ideal in $R/x_dR$.
Further, as $x_d$ is regular on $R/J$, $x_dR$ cannot contain any generator of $J$ and so $\mu(\bar{J})=\mu(J)$. 
Thus we may reduce the value of $d$.
\end{remark}

\begin{prop}
Let $n=2$ and assume the setting of Notation~\ref{notation}.
Then $\mu(J)\leq 3$.
\end{prop}

\begin{proof}
Here we may choose a regular system of parameters $x_1,x_2,\dots, x_d$ such that $x_1,x_2\in J$ and $x_{g+1},\dots,x_d$ form a system of parameters for $R/J$.
A standard basis for $J$ will be $x_1,x_2,f_3$ where $LT(f_3)=x_3^m$ where $m$ is the smallest $j$ such that there exists $s\in J$ with $LT(s)=x_3^j$.
\end{proof}

\begin{prop}
Let $n=3$ and assume the setting of Notation~\ref{notation}.
Then $\mu(J)\leq 6$.
\end{prop}

\begin{proof}
By Remark~\ref{remark}, it suffices to prove the proposition when $d=3$.
To handle $d=3$, let $x_1,x_2,x_3$ be a regular system of parameters for $R$.
If $\{r^{n-1}\mid r\in\m\}\nsubseteq J+\m^n$, we may apply Lemma~\ref{dimension reduction} to find an element $z$ and we may choose $x_3$ to be that $z$.
Then $\mu ((J+\m^3+x_3R)/(\m^3+x_3R))\geq 3$.
As $\dim \bar{\m}^2/\bar{\m}^3=3$ in a two dimensional regular local ring, we see that $x_1^2,x_1x_2,x_2^2\in J+\m^3+x_3R$.
Then $\m^3\subseteq \m J+\m^4+x_3R$ and by Proposition~\ref{Lemma 487}, $\mu(J)\leq \begin{pmatrix} 4\\2\end{pmatrix}=6$.

Finally, if $\{r^{n-1}\mid r\in\m\}\subseteq J+\m^n$, we have elements $f_1=x_1^2+h_1,f_2=x_2^2+h_2,f_3=x_3^2+h_3\in J$ with each $h_i\in \m^3$.
These three elements will constitute part of a standard basis for $J$.
If there exists a standard basis element $f_4\in J$ with $LT(f_4)=x_1x_2x_3$, then $J=(f_1,f_2,f_3,f_4)R$ and $\mu(J)\leq 4$.
Otherwise, all leading terms of standard basis elements will lie outside $\m^3$ and we see $\mu(J)\leq 6$.
\end{proof}

\begin{prop}
Let $n=4$ and assume the setting of Notation~\ref{notation}.
Then $\mu(J)\leq 10$.
\end{prop}

\begin{proof}
By Remark~\ref{remark}, we may assume $d=\dim R \leq 4$.
If $d=4$, we choose an $n$-compliant system of parameters $x_1,x_2,x_3,x_4$ for $R$.
Suppose $\{r^{n-1}\mid r\in\m\}\subseteq J+\m^n$.
Then we have elements $f_1=x_1^3+h_1,f_2=x_2^3+h_2,f_3=x_3^3+h_3\in J$ with each $h_i\in \m^4$.
These three elements will constitute part of a standard basis for $J$.
We can count the size of a standard basis by counting the number of leading terms and Proposition~\ref{1.4} tells us that it is sufficient to count the $x_1^{e_1}x_2^{e_2}$ which are factors of the leading terms.
The only possibilities are $x_1^3,x_1^2x_2^2,x_1^2x_2,x_1^2,x_1x_2^2,x_1x_2,x_1,x_2^3,x_2^2,x_2,1$.
These are eleven in all and so it suffices to show all cannot occur simultaneously to show $\mu(J)\leq 10$.
To do this, we note that $x_1^2x_2^2,x_1^2x_2,x_1^2,x_1,1$ cannot all occur since we would need five different values for $e_3$ and $0\leq e_3\leq 3$.
Thus we may again employ Lemma~\ref{dimension reduction} and Remark~\ref{remark} to reduce to the case $d=3$.

Let $x_1, x_2,x_3$ be an $n$-compliant system of parameters for $R$ and suppose $f_1,\dots,f_m\in J$ are linearly independent modulo $\m^4$.
We set $\phi_i=\Phi(f_i)$ for each $i$.
It would violate the hypothesis if $(\phi_1,\dots,\phi_m)k[x_1,x_2,x_3]$ were contained in a height one (principal) ideal and so we can easily choose $f_1,\dots,f_m$ so that $\phi_1,\phi_2$ generate a height two ideal.
Throughout the process, we will often change our system of parameters by replacing $x_j$ with $x_j'=x_j-\alpha x_i$ with $\alpha$ a unit in $A$ and $i\neq j$.
As noted earlier, it will be easy to do so while keeping the parameters $n$-compliant.
We claim that we may choose our parameters $x_1,x_2,x_3$ so that $\mu((J+x_3R+\m^4)/(x_3R+\m^4))\geq 3$.
Assume this is false.
Let $V=J+\m^4/\m^4$.
As the height of $(\phi_1,\dots,\phi_m,x_j)R$ is three for any permissible parameter, we must then have $\mu((J+x_jR+\m^4)/(x_jR+\m^4))=2$ for any parameter $x_j$.
First we note that if $x_1x_2x_3\in V$ and $x_1,x_2,x_3'=x_3-\alpha x_1$ and $x_1,x_2,x_3''=x_3-\beta x_2$ are both $n$-compliant systems of parameters, then either $x_1x_2x_3'\notin V$ or $x_1x_2x_3''\notin V$.
Hence $x_1x_2x_3\notin V$ for a general choice of parameters.
To see this simply note that otherwise we have $x_1^2x_2,x_1x_2^2\in V$ and so these two elements comprise a basis for $(J+x_3R+\m^4)/(x_3R+\m^4)$.
This contradicts the fact that $J\nsubseteq (x_1,x_3)R+\m^4$.

Using our standard change of parameters, we may adjust $x_2,x_3$ so that $x_1^3\in S(\phi_1)$.
Then we subtract multiples of $f_1$ from each $f_i$, $i>1$ so that $x_1^3\notin S(\phi_i)$, $i\neq 1$.
Similarly we adjust $x_1,x_3$ to force $x_2^3\in S(\phi_2)$; note this does not change any $x_1^3$ terms.
Again subtract multiples of $f_2$ from the others so that only $\phi_2$ has $x_2^3$ in its support.
We do the same to get $x_3^3\in S(\phi_3)$ and in the support of no other $\phi_i$.
As noted above, if necessary we may adjust $x_3$ further so that $x_1x_2x_3\notin V$.
The assumption $\mu((J+x_1R+\m^4)/(x_1R+\m^4))=2$ forces $x_1$ to divide some linear combination of $\phi_2,\phi_3,\phi_4$ and it quickly follows that $x_1$ divides $\phi_4$.
Repeating with $x_2,x_3$ gives $\phi_4=x_1x_2x_3$.
This contradiction proves the claim.

As $10= \begin{pmatrix} g+n-2\\ g-1 \end{pmatrix}$ when $g=3$ and $n=4$, we can complete the proof by applying Proposition~\ref{Lemma 487}  provided we can show $(x_1,x_2)^4R\subset \m J+\m^5+x_3R=\m(J+\m^4)+x_3R$.
This comes down to showing $(x_1,x_2)^4k=(x_1,x_2)\bar{V}$ where $\bar{V}$ is $V$ modulo $x_3M_2\cap V$.
The setup is then that we have three linearly independent polynomials $h_1,h_2,h_3\in (x_1,x_2)^3$ which do not have a common factor modulo $\m^4$ and want to show $(x_1,x_2)(h_1,h_2,h_3)R=(x_1,x_2)^4R$.
Suppose not.
Then $(x_1,x_2)^4/(x_1,x_2)^5$ is a five dimensional vector space and $x_1(h_1,h_2,h_3)/\m^5$, $x_2(h_1,h_2,h_3)/\m^5$ are three dimensional subspaces.
Our assumption forces the intersection to be dimension at least two and so, after a linear adjustment of $h_1,h_2,h_3$, there exist $s_1,s_2\in (h_1,h_2,h_3)R$ with $x_1\bar{h}_1=x_2\bar{s}_1,x_1\bar{h}_2=x_2\bar{s}_2$.
Then $(\bar{h}_1,\bar{h}_2)\cap(\bar{s}_1,\bar{s}_2)\neq (0)$.
We can adjust our notation again so that $\bar{h}_2$ is in the intersection and then adjust $h_1,h_3$ accordingly so that $x_1\bar{h}_1=x_2\bar{h}_2, x_1\bar{h}_2=x_2\bar{h}_3$.
It follows that there exists $t\in\m$ such that $\bar{h}_1=x_2^2\bar{t}, \bar{h}_2=x_1x_2\bar{t}, \bar{h}_3=x_1^2\bar{t}$.
However this contradicts the hypothesis that $h_1,h_2,h_3$ do not have a common factor modulo $\m^4$.
This completes the proof.
\end{proof}

\begin{remark}
The final proposition gives a bound one higher than desired.
In truth, the fact that an ideal is perfect is a very strong assumption and we are using little of that strength in this article.
The author is inclined to believe the optimal result is in fact true, but it seems clear that proving it will require techniques for exploiting perfectness that are not employed in this article.

The proof of the proposition is unconventional.  Many cases are dealt with and, in each case that is dismissed, the bound of $15$ is achieved.
It is only the final case discussed at the very end of the proof where $16$ is a possibility.
This completes the proof.
We note that, only in this final case do we allow for the possibility of sixteen generators.
\end{remark}

\begin{prop}
Let $n=5$ and assume the setting of Notation~\ref{notation}.
Then $\mu(J)\leq 16$.
\end{prop}

\begin{proof}
Let $x_1,\dots,x_d$ be a $5$-compliant system of parameters.
By Lemma~\ref{dimension reduction} and Remark~\ref{remark}, we can reduce to three cases:
\begin{enumerate}
  \item $d=3$
  \item $d=4$
  \item $d=5$ and $x_1^4,x_2^4,x_3^4\in \Phi_5(J)$.
\end{enumerate}
We shall dispose of Case 3 first as it is the easiest.
We mimic part of the proof of the previous result.
We have elements $f_1=x_1^4+h_1,f_2=x_2^4+h_2,f_3=x_3^4+h_3\in J$ with each $h_i\in \m^5$.
These three elements will constitute part of a standard basis for $J$.
(Recall we chose our order on monomials to respect degree to the extent that, for example, $LT(x_3^4+x_1^3x_2^2)=x_3^4$.)
We can count the size of a standard basis by counting the number of leading terms and Proposition~\ref{1.4} tells us that it is sufficient to count the $x_1^{e_1}x_2^{e_2}$ which are factors of the leading terms.
The only possibilities are $x_1^4, x_2^4,x_1^{e_1}x_2^{e_2}$ for $0\leq e_1,e_2\leq 3$.
These are eighteen in all and so it suffices to show that at most fifteen can occur simultaneously.
To do this, we note that $0\leq e_3\leq 4$ means we can have at most five different values for $e_3$ and so the longest possible well-ordered chain of $x_1^{e_1}x_2^{e_2}$ contains five elements.
Consider the chains $x_1^3x_2^3,x_1^3x_2^2,x_1^3x_2,x_1^3,x_1^2,x_1,1$ and
$x_1^3x_2^3,x_1^2x_2^3,x_1x_2^3,x_2^3,x_2^2,x_2,1$.
As each contains seven elements, at least two must be omitted from each chain.
As $1$ corresponds to $x_3^4$, it cannot be omitted.
The two chains then have only one remaining element in common and so omitting two from each forces us to omit at least three elements, reducing the number of possibilities to fifteen as desired.

Next we consider the $d=3$ case.
Suppose we can choose our parameters $x_1,x_2,x_3$ so that $\mu((J+x_3R+\m^5)/(x_3R+\m^5))\geq 4$.
Here we claim that $(x_1,x_2)^5R\subset \m J+\m^6+x_3R$.
Assuming the claim, we may employ Proposition~\ref{Lemma 487} with $\ell=5$ to get $\mu(J)\leq \begin{pmatrix} 6 \\ 2 \end{pmatrix}=15$ to prove the result in this situation.
To prove the claim, we first note that it is obvious if $\mu((J+x_3R+\m^5)/(x_3R+\m^5))=5$.
So it comes down to showing that if $K$ is a four dimensional subspace of $\tilde{M}_4$, the homogeneous degree $4$ elements of $k[x_1,x_2]$, which is not contained in $y\tilde{M}_3$ for any element $y$, the dimension of $x_1K+x_2K$, viewed as a subspace of $\tilde{M}_5$, equals $6$.
Since $x_1K,x_2K$ are both four dimensional, this will happen unless the intersection is three dimensional as $4+4-5=3$.
However, in that case, we have three dimensional subspaces $K_1,K_2$ of $K$ such that $x_1K_1=x_2K_2$.
As $K\nsubseteq x_1\tilde{M}_3$, we must have $\dim K\cap x_1\tilde{M}_3=3$ and so it must be that $K\cap x_1\tilde{M}_3=K_2$.
Another dimension argument tells us $K\cap x_1^3\tilde{M}_1\neq (0)$ and so there exists an element $x_1^3u\in K$.
Now $K\cap x_1\tilde{M}_3=K_2$ forces $x_1^2x_2u, x_1x_2^2u, x_2^3u\in K$ and $K=u\tilde{M}_3$, a contradiction which proves the claim.
So we have reduced to the case $\mu((J+x_3R+\m^5)/(x_3R+\m^5))\leq 3$ for every choice of $x_3$.

Let $f_1,\dots,f_5\in J$ constitute all or part of a generating set for $J+\m^5/\m^5$.
Let $\phi_i=\Phi_5(f_i)$ be the corresponding degree 4 homogeneous elements of $k[x_1,x_2,x_3]$, which we regard as a vector space over $k$ and denote as $M_4$.
We may write $\phi_i=\sum_{\ell=0}^4\phi_{i\ell}x_3^\ell$ where $\phi_{i\ell}$ is a homogeneous polynomial in $k[x_1,x_2]$ of degree $4-\ell$.
Adjusting the generators if necessary, we may assume the nonzero $\phi_{i0}$ are linearly independent and choose our original parameters $x_i$ so that the number of nonzero $\phi_{i0}$ is maximized.

By the reduction we have just achieved, we may assume $\phi_{40}=\phi_{50}=0$.
We also can suppose $\phi_{10}, \phi_{20}, \phi_{30}$ are nonzero.
If this is impossible, the proof given below is greatly simplified.
Making the change of variables $x_3'=x_3-\alpha x_1$ and writing $\phi_i=\sum_{\ell=0}^4\phi_{i\ell}'(x_3')^\ell$,
we see that $\phi_{i0}'=\sum_{\ell=0}^4\alpha^\ell \phi_{i\ell}x_1^\ell$.
Now $\phi_{10}',\phi_{20}',\phi_{30}',\phi_{40}'$ must be linearly dependent.
As $\tilde{M}_4$ is five dimensional, this says that all of the $4\times 4$ minors of a certain $4\times 5$ matrix are zero.
Now each of these minors is a polynomial in $\alpha$ and since they must vanish for all but finitely many $\alpha$ and $k$ is infinite, the polynomials must be identically zero.
Let $m$ be the smallest integer such that $\phi_{4m}\neq 0$. 
It is easy to see that the coefficients of $\alpha^j$ for $j<m$ are trivially zero and the coefficients of $\alpha^m$ all vanish exactly when $\phi_{10},\phi_{20},\phi_{30},x_1^m\phi_{4m}$ are linearly dependent, i.e., when $x_1^m\phi_{4m}$ is in the vector space with basis $\phi_{10},\phi_{20},\phi_{30}$.
Similarly we see that $x_2^m\phi_{4m}$ is also in that vector space and in fact so is $(x_1+\beta x_2)^m\phi_{4m}$ for infinitely many choices of $\beta$.
Likewise, if $q$ is minimal such that $\phi_{5q}\neq 0$, that space also contains $x_1^q\phi_{5q}$ and $x_2^q\phi_{5q}$.
This forces  $x_1^m\phi_{4m}, x_2^m\phi_{4m}, x_1^q\phi_{5q}, x_2^q\phi_{5q}$ to be a linearly dependent set and this fact will be our primary resource.
It should be noted that if $q=m$ and $\phi_{5q}$ is a linear multiple of $\phi_{4m}$, we can subtract a multiple of $f_4$ from $f_5$ and reduce to the case $m<q$.
Then our four elements will span the three dimensional space and we will always have a dependence relation
$$(Ax_1^m+Bx_2^m)\phi_{4m}=(Cx_1^q+Dx_2^q)\phi_{5q}$$

We complete the proof by considering the different possibilities for $m$ and $q$.  
By symmetry, we may assume $m\leq q$.
We shall consider six cases which exhaust all possibilities:  
\begin{enumerate}[(i)]
  \item $m+q<4$
  \item $q=4$
  \item $m=q=2$
   \item $m=q=3$
  \item $m=2,q=3$
  \item $m=1,q=3$
\end{enumerate}

Case (i):  Here the dependence relation forces $\phi_{4m},\phi_{5q}$ to have a nontrivial common factor $y$.
But then $J\subset (x_3,y)R+\m^5$, a contradiction, and so Case (i) cannot occur.

Case (ii):  Here $\phi_5=x_3^4$ and so $LT(f_5)=x_3^4$.
We see $x_1^4,x_2^4\in (\phi_{10},\phi_{20},\phi_{30})K$ and so we can rearrange $f_1,f_2,f_3$ so that $\phi_{10}=x_1^4=LT(f_1)$ and $\phi_{20}=x_2^4=LT(f_2)$.
The situation is now identical to what we had in Case (3) above and so we see $\mu(J)\leq 15$.

Case (iii):  We have $(Ax_1^2+Bx_2^2)\phi_{42}=(Cx_1^2+Dx_2^2)\phi_{52}$.
As $\phi_{42}, \phi_{52}$ must be relatively prime, we may assume $\phi_{42}=Cx_1^2+Dx_2^2$ and $\phi_{52}=Ax_1^2+Bx_2^2$.
Replacing $f_4,f_5$ and $f_1,f_2,f_3$ by appropriate linear combinations, we can reduce to the case $\phi_{42}=x_1^2$, $\phi_{52}=x_2^2$, $\phi_{10}=x_1^4$, $\phi_{20}=x_1^2x_2^2$, and $\phi_{30}=x_2^4$.
As $(x_1,x_2)(x_1^4,x_1^2x_2^2,x_2^4)=(x_1,x_2)^5$, we may invoke Proposition~\ref{Lemma 487} with $\ell=5$ to get $\mu(J)\leq 15$.

Case (iv):  Here $(Ax_1^3+Bx_2^3)\phi_{43}=(Cx_1^3+Dx_2^3)\phi_{53}$ gives $x_1^3(A\phi_{43}-C\phi_{53})=x_2^3(D\phi_{53}-B\phi_{43})$.
As $\phi_{43},\phi_{53}$ are linear, this forces $A\phi_{43}-C\phi_{53}=0$, which cannot happen as $\phi_{43}, \phi_{53}$ are linearly independent.
Thus Case (iv) is impossible.

Case (v):  This is another case which cannot occur.  
If $\charac k\neq 2$, $m=2$ is in fact impossible.
By varying $\beta$ in $(x_1+\beta x_2)^m\phi_{4m}$, we see that $x_1^2\phi_{42},x_1x_2\phi_{42},x_2^2\phi_{42}$ are all in $(\phi_{10},\phi_{20},\phi_{30})k$ and so form a basis for that vector space.
This contradicts $J\nsubseteq (x_3,\phi_{42})R+\m^5$.
Similarly $q=3$ cannot occur if $\charac k\neq 3$ and $k$ cannot simultaneously have two distinct characteristics.

Case (vi):  We have $(Ax_1+Bx_2)\phi_{41}=(Cx_1^3+Dx_2^3)\phi_{53}$.
As $\phi_{41},\phi_{53}$ are relatively prime, we may assume $\phi_{41}=Cx_1^3+Dx_2^3$ and $\phi_{53}=Ax_1+Bx_2$.
At this point, we will no longer need compatible parameters and we replace $x_1,x_2$ by new parameters $x_1',x_2'$ where $x_2'=Ax_1+Bx_2$ and $x_1'=x_1$ unless $B=0$, in which case we set $x_1'=x_2$.
As usual, we drop the primes as we no longer need the former parameters.
Then $\phi_{53}=x_2$ and $\phi_{41}=C'x_1^3+D'x_2^3$ with $C'\neq 0$.
We may rearrange $f_1,f_2,f_3$ so that $\phi_{10}=(C'x_1^3+D'x_2^3)x_1$, $\phi_{20}=x_1^3x_2$, and $\phi_{30}=x_2^4$.
Then $f_1,f_2,f_3,f_4,f_5$ will all be part of a standard basis and have respective leading terms $x_1^4,  x_1^3x_2, x_2^4, x_1^3x_3, x_2x_3^3$.
Again we can count the size of a standard basis by counting the number of leading terms and Proposition~\ref{1.4} tells us that it is sufficient to count the $x_1^{e_1}x_2^{e_2}$ which are factors of the leading terms.
Clearly, except for $x_1^2x_2^3$, all possibilities require $e_1+e_2\leq 4$ and so there are at most $16$; it only remains to show that all cannot occur simultaneously.
Consider the chain $x_2,x_2^2,x_2^3,x_1x_2^3,x_1^2x_2^3$.
For this entire chain to occur, we would need a descending sequence of five values of $e_3$.
However, as the first value of $e_3$ is known to be $3$, this is impossible and the proof of Case (1) is complete.

Finally we deal with the case $d=4$.
First we claim that we can find a $5$-compatible system of parameters $x_1,x_2,x_3,x_4$ such that $V/(x_i,x_j)M_3\cap V$ has dimension $3$ for every $i\neq j$.
To prove the claim, we first make the easy observation that we can assume the dimension is at most three.
If there is an element $x_i$ such that $\dim V/x_iM_3\cap V=5$, we can reduce to the $d=3$ case by modding it out.
With no such element, the only way $\dim V/(x_i,x_j)M_3\cap V>3$ is if $x_iM_3\cap V=x_jM_3\cap V$ and for fixed $x_i$, this can happen for at most three values of $x_j$.
This is easily avoided.
The next step is to observe that if $\dim V/(x_3,x_4)M_3\cap V=3$ and $x_4'=x_4+\alpha x_2$, $\dim V/(x_3,x'_4)M_3\cap V=3$ for all but finitely many choices of $\alpha$.
To see this, suppose $f_1,f_2,f_3$ are linearly independent module $(x_3,x_4)M_3$.
Then a certain $3\times 3$ minor of a $3\times 5$ matrix is nonzero.
(The columns correspond to $x_1^4,x_1^3x_2,x_1^2x_2^2,x_1x_2^3,x_2^4$.)
The alteration changes the entries of the matrix to polynomials in $\alpha$ with the same constant terms as the original matrix.
The new minor is then a polynomial in $\alpha$ with nonzero constant term and so can vanish for at most finitely many values of $\alpha$.
What this means is that to get $\dim V/(x_i,x_j)M_3\cap V=3$ for all pairs $i\neq j$, we can fix the pairs one at a time with the assurance that repairing one pair by an alteration of this kind can be done without spoiling those that are already good.
We will now show that we can repair the pair $(3,4)$.

Suppose $\dim V/(x_3,x_4)M_3\cap V<3$.
We can find $\phi_1\in x_1M_3\cap V$ and, as earlier, we can adjust $x_2,x_3,x_4$ to get $x_1^4\in S(\phi_1)$. Similarly we can get $\phi_2\in x_2M_3\cap V$ and by adjusting $x_3,x_4$, obtain $\phi_2\notin (x_3,x_4)M_3$.
Now we have $\dim V/(x_3,x_4)M_3\cap V\geq 2$ and we may assume it is exactly two since otherwise we have established the claim.
Since $V$ is spanned by $(x_3,x_4)M_3\cap V$ and $\{\phi_1,\phi_2\}$ and is not contained in $(x_1,x_3,x_4)M_3$, we must have $x_2^4\in S(\phi_2)$.
The same setup remains valid if $x_4$ is replaced by $x_4'=x_4+\alpha x_1$ for all but finitely many choices of $\alpha$ modulo $k$.
If $\dim V/x_4M_3\cap V = 4$, we work modulo $x_4$ and we are in a case which resembles the $d=3$ case except that we only have $f_1, f_2, f_4, f_5$ instead of $f_1, f_2, f_3, f_4, f_5$.
Effectively $\phi_{30}=0$ and we saw that led to an immediate contradiction.

Thus symmetrically we may assume $3=\dim V/x_4M_3\cap V =\dim V/x_3M_3\cap V =\dim V/(x_4')M_3\cap V $.
Now, for any parameter $y$, $yM_3\cap V=zM_3\cap V$ is possible for only finitely any values of $z$ (up to unit multiple of course).
Thus we may choose $x_3,x_4,x_3'=x_3+\beta x_4$ such that $x_3M_3\cap V, x_4M_3\cap V, x_3'M_3\cap V$ are all distinct.
As $\dim (x_3M_3\cap x_4M_3\cap V)=\dim (x_3M_3\cap V)-1$ and  $\dim (x_3M_3\cap x_3'M_3\cap V)=\dim (x_3M_3\cap V)-1$, we have $x_3M_3\cap V=x_3x_4M_2\cap V+x_3x_3'M_2\cap V$.
It follows that if $x_3\theta\in V$, $\theta\in (x_4,x_3')M_2=(x_3,x_4)M_2$.
Thus, modulo $x_3M_2$, $\theta\in x_4M_2$.
However, if we choose any $x_4'$ such that $x_3M_3\cap V\neq x_4'M_3\cap V$, the same argument yields $\theta\in x_4'M_2$ modulo $x_3M_2$.
This is impossible for more than three values of $x_4'$ unless $\theta$ is congruent to zero and so we get $\theta\in x_3M_2$.
Thus $x_3M_3\cap V=x_3^2M_2\cap V$ and by symmetry, $x_4M_3\cap V=x_4^2M_2\cap V$.
It follows that $x_3M_3\cap x_4M_3\cap V=x_3^2x_4^2k$ and $\dim V=5$.
As $\dim x_3M_3\cap V=2$ and this subspace contains $x_3^2(x_3')^2,x_3^2x_4^2,x_3^2(x_4')^2$, we have a clear contradiction.
We have now shown that we can find a $5$-compatible system of parameters $x_1,x_2,x_3,x_4$ such that $V/(x_i,x_j)M_3\cap V$ has dimension $3$ for every $i\neq j$ and we do so.

Next, if $\dim V>5$, we get $\dim x_3M_3\cap V=\dim V-4>1$.
Using the same argument as above, we get $x_3M_3\cap V=x_3^2M_2\cap V$ and assorted symmetric results.
As $x_3M_3\cap x_4M_3\cap V=x_3^2x_4^2k$, we get $\dim V=6$ and so $\dim x_3M_3\cap V=2$.
But $x_3^2(x_3')^2,x_3^2x_4^2,x_3^2(x_4')^2\in x_3M_3\cap V$, a contradiction which forces $\dim V=5$.
For any parameter $y$, there is a maximal integer $q_y$ such that $yM_3\cap V=y^{q_y}M_{4-q_y}\cap V$.
Let $q$ be the minimal value of $q_y$ if $y$ ranges over all generic parameters.
We now may choose a basis for $V$ consisting of $\phi_1=x_1^q\theta_1, \phi_2=x_2^q\theta_2, \phi_3=x_3^q\theta_3, \phi_4=x_4^q\theta_4, \phi_5$ with $\theta_1\notin x_1M_{4-q-1}$ and $\phi_\ell\notin (x_i,x_j)M_3$ unless $\ell$ equals either $i$ or $j$.
The last part is forced by $\dim V/(x_i,x_j)M_3\cap V=3$.

For all but finitely any $\alpha$, we can find an element $\sigma_\alpha=(x_1+\alpha x_2)^qh_\alpha\in V$.
As $\sigma_\alpha\in (x_1,x_2)M_3\cap V$ and $\phi_3,\phi_4,\phi_5$ are linearly independent modulo $(x_1,x_2)M_3\cap V$, we have $\sigma_\alpha=A_\alpha\phi_1+B_\alpha\phi_2$.
Avoiding finitely many choices of $\alpha$ so that $x_1+\alpha x_2$ does not divide $\theta_2$, we have $A_\alpha\neq 0$ and, replacing $\sigma_\alpha$ by a constant multiple, we may assume $A_\alpha =1$.
We now have $$(x_1+\alpha x_2)^qh_\alpha=x_1^q\theta_1+B_\alpha x_2^q\theta_2.$$
If $q>1$ but is not a power of the characteristic, we can obtain a quick contradiction.
Here $(x_1+\alpha x_2)^q$ will contain terms not divisible by $x_1^q$ or $x_2^q$.
This forces $h_\alpha\in (x_1,x_2)M_{4-q-1}$ and this in turns forces $\theta_1\in (x_1,x_2)M_{4-q-1}$.
By symmetry, $\theta_1\in (x_1,x_3)M_{4-q-1}$ and $\theta_1\in (x_1,x_4)M_{4-q-1}$.
However, the intersection of these three spaces is just $x_1M_{4-q-1}$ and this is impossible.
Hence we know $q=1$ or $q$ is a power of the characteristic of $k$ and so the equation above simplifies to 
$$x_1^qh_\alpha+\alpha^qx_2^qh_\alpha=x_1^q\theta_1+B_\alpha\x_2^q\theta_2.$$
Then $x_1^q(h_\alpha-\theta_1)=x_2^q(B_\alpha\theta_2-\alpha^qh_\alpha)$, which implies the existence of $\tau_\alpha$ such that  $h_\alpha-\theta_1=x_2^q\tau_\alpha$ and $B_\alpha\theta_2-\alpha^qh_\alpha=x_1^q\tau_\alpha$.
Combining these two equations to eliminate $h_\alpha$ gives $B_\alpha\theta_2-\alpha^q\theta_1=(x_1+\alpha x_2)^q\tau_\alpha$.
For all but finitely many $\alpha$, $B_\alpha\neq 0$ and we have $\theta_2\in (\theta_1, (x_1+\alpha x_2)^q)$ and so $\theta_2$ is a multiple of $\theta_1$.
Using symmetry and rescaling, we may assume $\theta_i=\theta_j$ for all $i,j$ and so $\phi_i=x_i^q\theta$ for all $i\leq 4$.

Next we may adjust $x_2,x_3,x_4$ so that $x_1^{4-q}\in S(\theta)$.
Then we may adjust $x_1,x_3,x_4$ so that $x_2^{4-q}\in S(\theta)$.
Note that this will not affect the $x_1^{4-q}$ term.
Similarly we complete the process and get $x_i^{4-q}\in S(\theta)$ for all $i$.
This gives $x_i^4\in S(\phi_i)$ for each $i$ and subtracting multiples of $f_1,f_2,f_3,f_4$ from $f_5$, we may also assume $x_i^4\notin S(\phi_5)$.

If $q=4$, we can reduce to Case (3) and so we are done.
For $q<4$, we view $S=k[x_1,x_2]$ as a homomorphic image of $k[x_1,x_2,x_3,x_4]$ in the natural way and, as earlier, view $\tilde{M}_i$ as the vector space of homogeneous forms of degree $i$.
Let $K$ be the subspace of $\tilde{M}_4$ spanned by the images of $\phi_1,\phi_2,\phi_5$.  ($\phi_3,\phi_4$ map to zero.)
So $K$ is three dimensional with basis $x_1^q\bar{\theta},x_2^q\bar{\theta},\bar{\phi_5}$.
We now handle the case where $q=2$.
Here it is true that $x_1K+x_2K=\tilde{M}_5$.
To see this, note that the left hand vector space is a subspace of the right side, which is itself a six dimensional vector space.
Hence we get the desired equality if the left side also has dimension six, something that will follow if $x_1^{q+1}\bar{\theta},x_1^qx_2\bar{\theta},x_1x_2^q\bar{\theta},x_2^{q+1}\bar{\theta},x_1\bar{\phi_5},x_2\bar{\phi_5}$ are linearly independent.
If not however, we get a dependence relation which can be written as $$(Ax_1^{q+1}+Bx_1^qx_2+Cx_1x_2^q+Dx_2^{q+1})\bar{\theta}+(Ex_1+Fx_2)\bar{\phi_5}=0.$$
As $\bar{\theta}$ and $\bar{\phi_5}$ are relatively prime and $\bar{\theta}$ has degree two, this is impossible.
Thus we have $(x_1,x_2)^5R\subset \m J+\m^6+(x_3,x_4)R$ and by Proposition~\ref{Lemma 487} with $\ell=5$, we get $\mu(J)\leq 15$.

If $q=3$, the same proof tells us that $\bar{\phi_5}$ is a scalar multiple of $Ax_1^4+Bx_1^3x_2+Cx_1x_2^3+Dx_2^4$.
As $x_1^4,x_2^4\notin S(\phi_5)$, by symmetry, we may assume $\bar{\phi_5}=x_1^3x_2+\delta x_1x_2^3$.
As we can rescale $x_2$, we may also assume $\bar{\theta}=x_1+x_2$.
A basis for $x_1K+x_2K$ is then $x_1^5+x_1^4x_2, x_1^4x_2+x_1^3x_2^2, x_1^4x_2+\delta x_1^2x_2^3, x_1^2x_2^3+x_1x_2^4, x_1x_2^4+x_2^5$.
Moreover, if $g\in k[x_1,x_2]$ is any homogeneous polynomial of degree $6$, $\bar{\phi}_5$ divides $g$ modulo $x_1+x_2$ as $\bar{\phi}_5\equiv x_1^4$ and $g\equiv x_1^6$.
So $g=\sigma_1 \bar{\theta}+\sigma_2\bar{\phi}_5$ and since $\sigma_1$ has degree $5$, $\sigma_1\in (x_1^3,x_2^3)$ and so $g\in (\bar{\phi}_1,\bar{\phi}_2,\bar{\phi}_5)$.
It follows that $x_1^5,x_1^4x_2,x_1^3x_2^2,x_1^2x_2^3,x_1x_2^4,x_2^6\in LT(\m J)$.
Recall that we can count the size of a standard basis by counting the number of leading terms and Proposition~\ref{1.4} tells us that it is sufficient to count the $x_1^{e_1}x_2^{e_2}$ which are factors of the leading terms.
It is clear that, for any leading term in our standard basis, we must have either $e_1+e_2\leq 5$ or $e_1=0$ and $e_2=6$.
It is also clear that, with the single possible exception of $x_1^0x_2^5$, any standard basis element with leading term such that $e_1+e_2>4$ is superfluous.
As there are exactly fifteen possibilities with $e_1+e_2\leq 4$, we have already shown $\mu(J)\leq 16$.
However, $f_1,f_2,f_3,f_5$ are the only standard basis elements which have leading terms of degree $4$.
Since $LT(f_3)=x_1x_3^3$, there cannot be a standard basis element with leading term either $x_1^2x_3^e$ or $x_1x_2x_3^e$ and so $\mu(J)\leq 14$ and the $q=3$ case is handled.
It remains only to consider $q=1$.

The $q=1$ case begins similarly. 
We let $K$ be the three dimensional subspace of $k[x_1,x_2]$ spanned by $x_1\bar{\theta},x_2\bar{\theta},\bar{\phi}_5$.
We claim that the dimension of $(x_1,x_2)K$ is $5$ and the dimension of $(x_1,x_2)^2K$ is $7$.
The first space is spanned by $x_1^2\bar{\theta},x_1x_2\bar{\theta},x_2^2\bar{\theta},x_1\bar{\phi}_5, x_2\bar{\phi}_5$ and the second is spanned by $x_1^3\bar{\theta},x_1^2x_2\bar{\theta},x_1x_2^2\bar{\theta},x_2^3\bar{\theta},x_1^2\bar{\phi}_5,x_1x_2\bar{\phi}_5,x_2^2\bar{\phi}_5$.
The proofs that these sets are linearly independent are almost identical and we do the second one.
If the set is linearly dependent, we get an equation $$(Ax_1^3+Bx_1^2x_2+Cx_1x_2^2+Dx_2^3)\bar{\theta}=(Ex_1^2+Fx_1x_2+Gx_2^2)\bar{\phi}_5.$$
Since $\bar{\theta},\bar{\phi}_5$ are relatively prime of degrees $3,4$ respectively, this is impossible.

We can count the size of a standard basis by counting the number of leading terms and Proposition~\ref{1.4} tells us that it is sufficient to count the $x_1^{e_1}x_2^{e_2}$ which are factors of the leading terms.
From our analysis of subspaces, we see that there are no standard basis elements with $e_1+e_2\geq 7$.
The fact that the first subspace has dimension five guarantees that six possible leading terms of degree six will be multiples of smaller leading terms and so we can obtain at most one standard basis element with $e_1+e_2=6$ and that standard basis element will be superfluous.
We are only guaranteed that four of the leading terms of degree five will be multiples of smaller leading terms and so there are potentially two standard basis elements with $e_1+e_2=5$, one of which will be superfluous while the other will not,
The number of standard basis elements with $e_1+e_2\leq 4$ is of course bounded by $15$ by standard combinatorics.
Hence the number of non-superfluous standard basis elements is at most $16$.
\end{proof}

\end{document}